\documentclass[11pt]{amsart}
\usepackage{amssymb}
\newtheorem{lemme}{Lemma}
\newtheorem{coro}{Corollary}
\newtheorem{theorem}{Theorem}
\newtheorem{remark}{Remark}

\begin{document}

\title [ Logarithmically completely monotontic functions]{ Logarithmically completely monotontic functions related the $q-$gamma and the  $q-$digamma functions with applications\\}%

\author[ Khaled Mehrez]{KHALED MEHREZ }
\address{Khaled Mehrez. D\'epartement de Math\'ematiques ISSAT Kasserine, Universit\'e de Kairouan, Tunisia.}
 \email{k.mehrez@yahoo.fr}
\begin{abstract}
  In this paper we present several new classes of logarithmically completely monotonic
functions.  Our functions have in common that they are defined in terms of the  $q-$gamma and $q-$digamma functions. As an applications of this results, some  inequalities for the $q-$gamma and the $q-$digamma functions are established. Some of the given results generalized theorems due to Alzer and Berg and C.-P. Chen and F. Qi.
\end{abstract}
\maketitle
{\it keywords:} Completely monotonic functions, Logarithmically completely monotonic functions, $q$-gamma function, $q-$digamma function, Inequalities  \\
MSC (2010): 
\section{\textbf{Introduction}}
It   is   well-known   that   the   classical   Euler   gamma 
function may be defined by 
$$\Gamma(x)=\int_{0}^{\infty}t^{x-1}e^{-t}dt,$$
$\psi^{(k)}(x)$ for $k\in\mathbb{N}$ are called the polygamma functions. It is common knowledge that special functions 
 $\Gamma (x),\;\psi(x)$ and $\psi^{(k)}(x)$  for  $k\in\mathbb{N}$  are  fundamental  and  important 
and  have  much  extensive  applications  in  mathematical sciences.

The  $q-$analogue of  $\Gamma$  is defined [\cite{aaa}, pp. 493-496] for $x>0$ by 
\begin{equation}
\Gamma_q(x)=(1-q)^{1-x}\prod_{j=0}^{\infty}\frac{1-q^{j+1}}{1-q^{j+x}},\:0<q<1,
\end{equation}
and 
\begin{equation}
\Gamma_q(x)=(q-1)^{1-x}q^{\frac{x(x-1)}{2}}\prod_{j=0}^{\infty}\frac{1-q^{-(j+1)}}{1-q^{-(j+x)}},\:q>1.
\end{equation}
The  $q-$gamma  function  $\Gamma_q(z)$  has  the  following  basic 
properties:
\begin{equation}
\lim_{q\longrightarrow1^{-}}\Gamma_{q}(z)=\lim_{q\longrightarrow1^{+}}\Gamma_{q}(z)=\Gamma(z),
\end{equation}
and 
\begin{equation}\label{369}
\Gamma_{q}(z)=q^{\frac{(x-1)(x-2)}{2}}\Gamma_{\frac{1}{q}}(z).
\end{equation}
The $q-$digamma function $\psi_q$, the  $q-$analogue  of  the  psi  or  digamma  function $\psi$ is 
defined by 
\begin{equation}\label{int}
\begin{split}
\psi_q(x)&=\frac{\Gamma^{'}_q(x)}{\Gamma_q(z)}\\
&=-\ln(1-q)+\ln q \sum_{k=0}^{\infty}\frac{q^{k+x}}{1-q^{k+x}}\\
&=-\ln(1-q)+\ln q \sum_{k=1}^{\infty}\frac{q^{kx}}{1-q^{k}}\\
&=-\ln(1-q)-\int_{0}^{\infty}\frac{e^{-xt}}{1-e^{-t}}d\gamma_q(t),
\end{split}
\end{equation}
for $0<q<1$, where $d\gamma_q(t)$ is  a  discrete  measure  with 
positive  masses  $-\ln q$  at  the  positive  points $-k\ln q$ for $k\in\mathbb{N}$,  more accurately, (see \cite{34})
\begin{equation}
\gamma_q(t)=\sum_{k=1}^{\infty}\delta(t+k\ln q),\,0<q<1.
\end{equation}
For $q>1$ and $x>0$, the $q-$digamma function $\psi_q$ is defined by 
\begin{equation*}
\begin{split}
\psi_q(x)&=-\ln(q-1)+\ln q\left[x-\frac{1}{2}-\sum_{k=0}^{\infty}\frac{q^{-(k+x)}}{1-q^{-(k+x)}}\right]\\
&=-\ln(q-1)+\ln q\left[x-\frac{1}{2}-\sum_{k=1}^{\infty}\frac{q^{-kx}}{1-q^{-kx}}\right]
\end{split}
\end{equation*}
Krattenthaler and Srivastava \cite{ks} proved that $\psi_q(x)$ tends to $\psi(x)$ on letting $q\longrightarrow 1$ where $\psi(x)$ is the  the ordinary Psi (digamma) function. 
Before we present the main results of this paper we recall some definitions, which will be used in the
sequel.

A function $f$ is said to be completely monotonic on an interval $I$ if $f$ has derivatives of all orders on $I$ and 
\begin{equation}\label{hh}(-1)^n f^{(n)}(x)\geq 0,\end{equation}
for all $x\in I$ and $n\in\mathbb{N}_0=\mathbb{N}\cup\left\{ 0\right\} $, where  $\mathbb{N}$ the set of all positive integers.

If the inequality (\ref{hh}) is strict, then $f$ is said to be strictly completely monotonic function.
 
A positive function $f$ is s aid to be logarithmically completely monotonic on an interval $I$ if its logarithm $\ln f$ satisfies
$$(-1)^n \Big(\ln f(x)\Big)^{(n)}(x)\geq 0,$$
for all $x\in I$ and $n\in\mathbb{N}.$

A positive function $f$  is said to be logarithmically convex on an interval $I$, or simply log-convex, if its
natural logarithm $\ln f$ is convex, that is, for all $x, y \in I$ and $ \lambda\in [0, 1]$ we have
$$f\left(\lambda x+(1-\lambda)y\right)\leq\left[f(x)\right]^{\lambda}\left[f(y)\right]^{1-\lambda}.$$
We note that  every logarithmically completely monotonic function is log-convex.
From the definition of $\psi_q(x)$, direct differentiation, and the induction we get
\begin{equation}
(-1)^{n}\psi_{q}^{(n+1)}(x)=(\ln q)^{n+2}\sum_{k=1}^{n}\frac{k^{n+1}q^{kx}}{1-q^k}>0,
\end{equation}
which implies that the function $\psi_{q}^{'}$ is strictly completely monotonic function on $(0,\infty),$ for $q\in(0,1).$ The relation (\ref{369}) and the definition of the $q-$digamma function (\ref{int}) give 
\begin{equation}\label{1000}
\psi_q(x)=\psi_{\frac{1}{q}}(x)+\frac{2x-3}{2}\ln q,
\end{equation}
for $q>1$. Thus, implies that the function $\psi_q^{'}(x)$ is strictly completely monotonic in $(0,\infty)$ for $q>1,$ and consequently the function $\psi_q(x)$ is strictly completely monotonic on $(0,\infty).$

It is the aim of this paper to provide several new classes of logarithmically completely monotonic functions.  The
functions  we  study  have  in  common  that  they  are  defined  in  terms  of  $q-$gamma and $q-$digamma functions. In the next section we collect some lemmas. Our monotonicity theorems are stated and proved in sections 3. 

\section{\textbf{Useful lemmas}}

We begin this section with the following useful lemmas which are needed to
completes the proof of the main theorems.

The following monotonicity theorem is proved in  \cite{alz1}.

\begin{lemme}\label{l4} Let $n$ be a natural number and $c$ be a real number.  The function $x^{c}|\psi^n(x)|$ is decreasing on $(0,\infty)$ if and only if $c\leq n.$
\end{lemme}

\begin{lemme}\cite{kha1}\label{ll5}
Let $f,g:[a,b]\longrightarrow\mathbb{R}$ two continuous functions which
are differentiable on $(a,b).$ Further, let $g^{\prime}\neq0$ on
$(a,b).$ If $\frac{f^{\prime}}{g^{\prime}}$ is increasing (or decreasing)
on $(a,b),$ then the functions $\frac{f(x)-f(a)}{g(x)-g(a)}$ and
$\frac{f(x)-f(b)}{g(x)-g(b)}$ are also increasing (or decreasing)
on $(a,b).$ 
\end{lemme}

The next lemma is given in \cite{alz, salem}.

\begin{lemme}\label{l1} The function $\psi_q,\;q>0$ has a uniquely determined positive zero, which we denoted by $x_0=x_0(q)\in(1,2).$
\end{lemme}

A proof for the following lemma can be found in \cite{kha}.

\begin{lemme}\label{l2}
Let $f$ be a positive function. If $f^{'}$ is completely monotonic function, then $\frac{1}{f}$ is logarithmically completely monotonic function.
\end{lemme}

The following inequality for completely monotonic functions is due to Kimberling \cite{ki}.

\begin{lemme}\label{l3} If a function $f$, defined on $(0,\infty),$ is continuous and completely monotonic and maps
$(0,\infty )$ into $(0,1),$ then $\ln f$ is super-additive, that is for all $x,y>0$ we have 
$$\ln f(x)+\ln f(y)\leq \ln f(x+y) \;\textrm{or} \;f(x)f(y)\leq f(x+y).$$ 
\end{lemme}

\section{\textbf{The main results}}
In 1997, Merkle \cite{merkle} proved that the function $\frac{(\Gamma(x))^2}{\Gamma(2x)}$ is log-convex on $(0,\infty)$. Recently,  Alzer and Berg \cite{alz8},  presented a substantial generalization. They established that the function 
\begin{equation}
\frac{\left(\Gamma(ax)\right)^{\alpha}}{\left(\Gamma(bx)\right)^{\beta}},\; 0<b<a,\;(\textrm{or}\; 0<a<b)
\end{equation}
is completely monotonic on $(0,\infty)$, if and only if, $\alpha\leq0,\;\;(\textrm{or}\;\alpha\geq0)$ and $\alpha a=\beta b.$ 
The main objective of the next  Theorem  extend and generalize this results. An application of their result leads to sharp upper and lower bounds for $\frac{\Gamma_q^2(x)}{\Gamma_q(2x)}$ in terms of the $\psi_q$-function.

\begin{theorem}\label{tt1}Let $0<q< 1$ and $0<a<b.$ Then the function $
\frac{\left(\Gamma_{q}(ax)\right)^{\alpha}}{\left(\Gamma_{q}(bx)\right)^{\beta}}$ is logarithmically completely monotonic on $(0,\infty)$ if and only if $\alpha\geq0$ and $\alpha a=\beta b.$
\end{theorem}
\begin{proof} Let $q\in(0,1).$ Assume that the function $\frac{\left(\Gamma_{q}(ax)\right)^{\alpha}}{\left(\Gamma_{q}(bx)\right)^{\beta}}$ is logarithmically completely monotonic on $(0,\infty)$. By definition, this gives us that for all real $x>0$,
\begin{equation}\label{0101}
\left[\ln\frac{\left(\Gamma_{q}(ax)\right)^{\alpha}}{\left(\Gamma_{q}(bx)\right)^{\beta}}\right]^{'}=\alpha a\psi_{q}(ax)-\beta b\psi_{q}(bx)\leq0.
\end{equation}
Now we prove $\alpha\geq0.$ By contradiction, suppose that $\alpha<0$ such that $\alpha a=\beta b$. Since the function $\psi_q(x)$ is strictly increasing on $(0,\infty)$ we have 
\begin{equation}\label{0404}
\alpha a(\psi_{q}(ax)-\psi_{q}(bx))>0,
\end{equation}
for $x>0$ and $0<a<b.$ On the other hand, by inequality (\ref{0101}) we obtain that
\begin{equation}\label{0405}
\alpha a(\psi_{q}(ax)-\psi_{q}(bx))\leq 0.
\end{equation}
Inequalities (\ref{0404}) and (\ref{0405}) implies that the constant $\alpha\geq0.$ \\
On the other hand, using the fact 
$$\lim_{x\longrightarrow+\infty}\psi_{q}(x)=-\ln(1-q),$$
for $q\in(0,1)$, and equality (\ref{0101}) we obtain 
$$(\alpha a-\beta b)(-\ln(1-q))\leq0.$$
and consequently 
$$(\alpha a-\beta b)\leq0.$$
On the other hand, using the definition we get
\[
\left[\ln\frac{\left(\Gamma_{q}(ax)\right)^{\alpha}}{\left(\Gamma_{q}(bx)\right)^{\beta}}\right]^{(2)}=\alpha a^2\psi_{q}^{(1)}(ax)-\beta b^2\psi_{q}^{(1)}(bx)\geq0.\]
If $x\longrightarrow 0^+$, we obtain 
\begin{equation}\label{0406}
\alpha a^2-\beta b^2\geq0,
\end{equation}
since $0<\frac{1}{b}<\frac{1}{a}$ we have $\alpha a\geq \beta b$,  this inequality and inequality (\ref{0406}) leads to $\alpha a=\beta b.$

Now we show that the function $\frac{\left(\Gamma_{q}(ax)\right)^{\alpha}}{\left(\Gamma_{q}(bx)\right)^{\beta}}$ is logarithmically completely monotonic on $(0,\infty)$ for $\alpha,\beta\geq 0,\;0<a<b$ such that $\alpha a=\beta b.$ Let $n=1$, since the function $\psi_q(x)$ is strictly  increasing on $(0,\infty)$ we have 
\begin{equation}\label{9990}
(-1)\left(\ln\frac{\left(\Gamma_{q}(ax)\right)^{\alpha}}{\left(\Gamma_{q}(bx)\right)^{\beta}}\right)^{'}=\alpha a\left(\psi_{q}(bx)-\psi_{q}(ax)\right)\geq0.
\end{equation}
 For $n\geq1,$ we get
\begin{equation}
(-1)^{n+1}\left(\ln\frac{\left(\Gamma_{q}(ax)\right)^{\alpha}}{\left(\Gamma_{q}(bx)\right)^{\beta}}\right)^{(n+1)}=(-1)^n\alpha a\left(b^{n}\psi_{q}^{(n)}(bx)-a^{n}\psi_{q}^{(n)}(ax)\right).
\end{equation}
 Since the function $\psi_q^{'}(x)$ is strictly completely monotonic on $(0,\infty)$, we have for $n\in\mathbb{N}$ and $q>0$ we have 
\begin{equation}
|\psi^{(n)}_q(x)|=(-1)^{n+1}\psi^{(n)}_q(x).
\end{equation}
Thus
\begin{equation*}\label{400}
\begin{split}
(-1)^{n+1}x^{n}\left(\ln\frac{\left(\Gamma_{q}(ax)\right)^{\alpha}}{\left(\Gamma_{q}(bx)\right)^{\beta}}\right)^{(n+1)}&=(-1)^n\alpha ax^{n}\left(b^{n}\psi_{q}^{(n)}(bx)-a^n\psi_{q}^{(n)}(ax)\right)\\
&=(-1)^n\alpha a\left((xb)^{n}(-1)^n|\psi_{q}^{(n)}(bx)|-(ax)^n(-1)^n|\psi_{q}^{(n)}(ax)|\right)\\
&=\alpha a \Bigg( (ax)^n|\psi_{q}^{(n)}(ax)|-(bx)^n|\psi_{q}^{(n)}(bx)|\Bigg),
\end{split}
\end{equation*}
and the last expression is nonnegative by Lemma \ref{l4}. Hence, for $q\in(0,1)$ and $n\in\mathbb{N}$
$$(-1)^n\left(\ln\frac{\left(\Gamma_{q}(ax)\right)^{\alpha}}{\left(\Gamma_{q}(bx)\right)^{\beta}}\right)^{(n)}\geq0.$$
The proof of Theorem \ref{tt1} is complete.
\end{proof}
\begin{coro}Let $q>1$ and $0<a<b.$ If, $\alpha\geq 0$ and $\alpha a=\beta b.$, Then the function  $\frac{\left(\Gamma_{q}(ax)\right)^{\alpha}}{\left(\Gamma_{q}(bx)\right)^{\beta}}$ is logarithmically completely monotonic on $(0,\infty)$.
\end{coro}
\begin{proof}
Follows immediately by Theorem \ref{tt1} and equality (\ref{1000}).
\end{proof}
\begin{coro}Let $q>0$ and $0<a<b$, the following inequalities 
\begin{equation}\label{555}
\exp\left[\alpha a(x-x_{1})\left(\psi_{q}(ax_{1})-\psi_{q}(bx_1)\right)\right]\leq\frac{\left(\Gamma_{q}(bx_{1})\right)^{\beta}}{\left(\Gamma_{q}(ax_{1})\right)^{\alpha}}\frac{\left(\Gamma_{q}(ax)\right)^{\alpha}}{\left(\Gamma_{q}(bx)\right)^{\beta}}\leq1,
\end{equation}
holds for all $\alpha,\;\beta\geq 0$ such that $\alpha a =\beta b.$ and $x>x_1>0$ In particular, the following inequalities holds true for every integer $n\geq1:$
\begin{equation}\label{666}
\exp\left[2q(n-1)\frac{\ln q}{1-q}\right]\leq\frac{\Gamma_{q}^{2}(n)}{\Gamma_{q}(2n)}\leq1.
\end{equation}
\end{coro}
\begin{proof}
Let $q>0$ and $0<a<b$. We suppose that $\alpha,\;\beta\geq 0$ such that $\alpha a =\beta b$ and define the function $h_{\alpha,\beta}(q;x)$ by 
$$h_{\alpha,\beta}(q;x)=\frac{\left(\Gamma_{q}(bx_{1})\right)^{\beta}}{\left(\Gamma_{q}(ax_{1})\right)^{\alpha}}\frac{\left(\Gamma_{q}(ax)\right)^{\alpha}}{\left(\Gamma_{q}(bx)\right)^{\beta}}$$
where $0<x_1<x,$  and $H_{\alpha,\beta,q}(x)=\ln h_{\alpha,\beta}(q;x).$ Since the function $h_{\alpha,\beta}(q;x)$ is logarithmically completely monoyonic on $(0,\infty)$ for $\alpha\geq0$ and $\alpha a=\beta b,$ we conclude that the logarithmic derivative $
\frac{\left(h_{\alpha,\beta}(q;x)\right)^{'}}{h_{\alpha,\beta}(q;x)}$ is increasing on $(0,\infty).$ By Lemma \ref{ll5} we deduce that the function $\frac{H_{\alpha,\beta}(q;x)}{x-x_1}$ is increasing for all $0<x_1<x.$ By l'Hospital's rule and (\ref{0101}) it is easy to deduce that 
\[
\lim_{x\longrightarrow x_{1}}\frac{H_{\alpha,\beta}(q;x)}{x-x_{1}}=\alpha a(\psi_{q}(ax_{1})-\psi_{q}(bx_{1})),\]
from which follows the right side inequality of (\ref{555}). 

As $h_{\alpha,\beta}(q;x)$ is logarithmically completely monotonic on $(0,\infty)$, we deduce that $h_{\alpha,\beta}(q;x)$ is decreasing on $(0,\infty).$ The following inequality hold true for all $0<x_1<x:$
$$h_{\alpha,\beta}(q;x)\leq h_{\alpha,\beta}(q;x_1)=1,$$
we conclude the left side inequality of (\ref{555}).

Taking $\alpha=b=2,\;\beta=a=1$ and $x_1=1$ in (\ref{555}) and  using the recurrence formula of $\psi_q$ [\cite{gf}, p. 1245, Theorem 4.4]
\begin{equation}\label{07+}
\psi_{q}^{(n-1)}(x+1)-\psi_{q}^{(n-1)}(x)=-\frac{d^{n-1}}{dx^{n-1}}\left(\frac{q^{x}}{1-q^{x}}\right)\ln q
\end{equation}
we obtain the inequalities (\ref{666}).
\end{proof}

The main purpose of the next Theorem is to present monotonicity properties of the function 
\begin{equation}\label{753}
g_{\beta}(q;x)=\frac{1}{1+q}\left[\frac{\Gamma_{q^2}\left(x+\frac{1}{2}\right)}{\Gamma_{q^2}(x+1)}\right]^{2}\exp\left[\frac{\beta(1-q^2)q^{2x}}{2(1-q^{2x})}+\psi_{q}(2x)\right],
\end{equation}
where $q\in(0,1)$ and $x>0.$

It is worth mentioning that Ai-Jun Li and Chao-Ping Chen \cite{cpc} considered the function
\begin{equation}
g_{\beta}(x)=\frac{1}{2}\left[\frac{\Gamma\left(x+\frac{1}{2}\right)}{\Gamma(x+1)}\right]^{2}\exp\left[\frac{\beta}{2x}+\psi(2x)\right]
\end{equation}
which is a special case of the function $g_{\beta}(q;x)$ on letting $q\longrightarrow1$ and proved that $g_{\beta}(x)$ is logarithmically completely monotonic on $(0,\infty)$ if $\beta\geq\frac{13}{12}$. The objective of this Theorem is to generalize this result.

\begin{theorem} Let $q\in(0,1)$. The function $g_{\beta}(q;x)$ defined by (\ref{753}) is logarithmically completely monotonic on $(0,\infty)$ if $\beta\geq \frac{-13\ln q}{6(1-q^2)}.$
\end{theorem}
\begin{proof}
It is clear that 
\[\ln g_{\beta}(q;x)=2\ln\Gamma_{q^2}(x+\frac{1}{2})-2\ln\Gamma_{q^2}(x+1)+\psi_{q}(2x)+\frac{\beta(1-q^2)q^{2x}}{2(1-q^{2x})}-\ln(q+1).\]
Using the  $q-$analogue  of  Legendre's  duplication   formula (\ref{leg}) we get
\[\ln g_{\beta}(q;x)=2\ln\Gamma_{q^2}(x+\frac{1}{2})-2\ln\Gamma_{q^2}(x+1)+\frac{1}{2}\psi_{q^2}(x)+\frac{1}{2}\psi_{q^2}(x+\frac{1}{2})+\frac{\beta(1-q^2)q^{2x}}{2(1-q^{2x})}.\]
In view of  (\ref{07+}) and  (\ref{int}) we obtain that 
\begin{equation*}
\begin{split}
(-1)^n\left(\ln g_{\beta}(q;x)\right)^{(n)}&=(-1)^{n}\Bigg[2\psi_{q^{2}}^{(n-1)}(x+\frac{1}{2})-2\psi_{q^{2}}^{(n-1)}(x+1)+\frac{1}{2}\psi_{q^{2}}^{(n)}(x)+\frac{1}{2}\psi_{q^{2}}^{(n)}(x+\frac{1}{2})\\&
-\frac{\beta(1-q^{2})}{4\ln q}\left(\psi_{q^{2}}^{(n)}(x+1)-\psi_{q^2}^{(n)}(x)\right)\Bigg]\\
&=\frac{1}{2}\int_{0}^{\infty}\frac{e^{-xt}}{e^t-1}\Phi_{\beta,q}(t)d\gamma_{q^2}(t)
\end{split}
\end{equation*}
where
\begin{equation*}
\begin{split}
\Phi_{\beta,q}(t)&=\left(\frac{-\beta(1-q^{2})}{2\ln q}-1\right)te^{t}+(4-t)e^{\frac{t}{2}}+\frac{\beta(1-q^{2})}{2\ln q}t-4\\
&=\sum_{n=0}^{\infty}\frac{t^{n+1}}{n!}\left[\frac{-\beta(1-q^{2})}{2\ln q}-1-\frac{1}{2^{n}}\right]+4\sum_{n=1}^{\infty}\frac{t^{n}}{2^{n}n!}+\frac{\beta(1-q^{2})}{2\ln q}t\\
&=\sum_{n=1}^{\infty}\frac{t^{n+1}}{n!}\left[\frac{-\beta(1-q^{2})}{2\ln q}-1-\frac{1}{2^{n}}\right]+4\sum_{n=0}^{\infty}\frac{t^{n+1}}{2^{n+1}(n+1)!}-2t\\
&=\sum_{n=1}^{\infty}\frac{t^{n+1}}{n!}\left[\frac{-\beta(1-q^{2})}{2\ln q}-1-\frac{1}{2^{n}}+\frac{1}{(n+1)2^{n-1}}\right]
\end{split}
\end{equation*}
Since the $\max \Bigg(\frac{1}{2^{n}}+\frac{1}{(n+1)2^{n-1}}\Bigg)=\frac{1}{12},$ we conclude that $\beta\geq\frac{-13\ln q}{6(1-q^2)}.$ The proof is completed. 
\end{proof}
\begin{theorem}\label{t2}Let $q>0$, the function $\frac{1}{\psi_{q}(x)}$ is Logarithmically completely monotonic on $(x_{0},\infty).$
\end{theorem}
\begin{proof}Since the function $\psi_{q}^{'}$ is completely monotonic function on $(0,\infty),$ and  the function $\psi_{q}$ is increasing on $(0,\infty)$ and a uniquely determined zero on $(0,\infty)$ by Lemma \ref{l1}. We conclude that the function $\psi_q(x)>0$ for all $x>x_0,$ and consequently the function $\frac{1}{\psi_{q}(x)}$ is Logarithmically completely monotonic on $(x_{0},\infty),$ by Lemma \ref{l2}.
\end{proof}
\begin{coro} \label{co1} Let $q>0$ and $a>1$. The following inequality 
\begin{equation}\label{1}
\left[\psi_{q}(x)\right]^{\frac{1}{a}}\left[\psi_{q}(y)\right]^{1-\frac{1}{a}}\leq\psi_{q}\left[\frac{x}{a}+\left(1-\frac{1}{a}\right)y\right]
\end{equation}
holds for all $x>x_0$ and $y>x_{0}.$
In particular, 	the following inequality holds
\begin{equation}\label{010}
\left[\psi_{q}(2)\right]^{a-1}\leq\frac{\left[\psi_{q}(u+1)\right]^{a}}{\psi_{q}\left(a(x-1)+2\right)}
\end{equation}
for all $a>1$ and $u>\frac{-2}{a}+1.$
\end{coro}
\begin{proof}
 Let $q\in(0,1), a>1, x>x_0$ and $y>x_0$. By theorem \ref{t2} we obtain that the function $\frac{1}{\psi_{q}(x)}$ is log-convex on $(x_{0},\infty).$ Thus,
\[
\left[\psi_{q}(x)\right]^{\frac{1}{p}}\left[\psi_{q}(y)\right]^{\frac{1}{q}}\leq\psi_{q}\left[\frac{x}{p}+\frac{y}{q}\right],\]
where $p>1,\;q>1,\;\frac{1}{p}+\frac{1}{q}=1$. If $p=a$ and $q=\frac{a}{a-1}.$ Then we get the inequality (\ref{1}). Now, let $y=2$ and $x=a(u-1)+2$ we obtain the inequality (\ref{010}).
\end{proof}
\begin{remark}Replacing $u$ by $n\in\mathbb{N}$  and $a$ by $2$ in inequality (\ref{010}) and using the identity 
$$\psi_q(n+1)=\frac{\ln q}{1-q}\gamma_q-\ln q H_{n,q}$$
where  $\gamma_q=\frac{1-q}{\ln q}\psi_q(1)$ is the $q-$analogue of the Euler-Mascheroni constant \cite{salem4} and $H_{n,q}$ is the q-analogue of Harmonic number is defined by \cite{wei} as 
$$H_{n,q}=\sum_{k=1}^{n}\frac{q^k}{1-q^k},\;n\in\mathbb{N},$$
we obtain 
\begin{equation}
\psi_{q}^2(2)\psi_{q}\left(2n\right)\leq\left[\frac{\ln q}{1-q}\gamma_q-\ln q H_{n,q}\right]^{2},\;n\in\mathbb{N}.
\end{equation}
\end{remark}
\begin{theorem} Let $q>0$ The function $\Gamma_q(x)$ is logarithmically completely monotonic on $(0,x_0).$ So, the following inequality 
\begin{equation}\label{999}
\Gamma_{q}(x+1)\Gamma_{q}(y+1)\leq \Gamma_{q}(x+y+2)
\end{equation}
holds for all $x,y\in(0,1).$
\end{theorem}
\begin{proof} Proving by induction that
$$(-1)^{n}(\ln \Gamma_q(x))^{(n)}\geq0,\;\textrm{for all}\; n\in\mathbb{N}.$$ 
For $n=1$, we get 
$$(-1)(\ln \Gamma_q(x))^{(1)}=\psi_{q}(x)\geq0,\; \textrm{for all}\; x\in (0,x_0).$$
Suppose that  
$$(-1)^{k}(\ln \Gamma_q(x))^{(k)}\geq0,\;\textrm{for all}\;1\leq k\leq n\;\textrm{and}\; x\in (0,x_0).$$ 
Since the function $\psi_{q}^{'}$ is completely monotonic on $(0,\infty),$ for $q>0$ we get
$$(-1)^{n+1}(\ln \Gamma_q(x))^{(n+1)}=(-1)^{n}\psi^{(n+1)}_{q}(x)\geq0.$$

We  note that every logarithmically completely monotonic function is completely monotonic \cite{berg}, and the function $\Gamma_q(x)\leq1$ if and only if $1\leq x\leq2$. The above results imply that the function $\Gamma_q(x)$ is completely monotonic on $(0,x_0)$ and maps $(1,2)$ into $(0,1)$, applying Lemma \ref{l3}, we conclude the asserted inequality  (\ref{999}) .
\end{proof}


\begin{thebibliography}{99}
\bibitem{alz} H. Alzer, A. Z. Grinsphan, Inequalities for gamma and $q-$gamma functions, J. of App. Theory, 144 (2007) 67–83.
\bibitem{alz1} H. Alzer, "Mean-value inequalities for the polygamma functions," Aequat. Math. 61 (2001), 151–161.
\bibitem{alz7} H. Alzer, A.Z. Grinshpan, Inequalities for the gamma and q-gamma functions, Journal of Approximation Theory
144 (2007) 67–83.
\bibitem{alz8} H. Alzer and C. Berg, "Some classes of completely monotonic functions II", Ramanujan J (2006) 11: 225–248.
445–460.
\bibitem{aaa} G.  E.  Andrews,  R.  A.  Askey,  and  R.  Roy,  Special  Functions, 
Cambridge University Press, Cambridge, 1999
\bibitem{ar}R. Askey, The $q-$Gamma and $q-$Beta functions, Applicable Analysis: An International Journal, 2013.
\bibitem{berg} C. Berg, Integral representation of some functions related to the gamma function, Mediterranean Journal of Mathematics 1(4)
(2004) 433–439.
\bibitem{kha} J. El Kamel, K. Mehrez, A Function Class  of strictly positive definite and logarithmically completely monotonic functions related to the modified Bessel functions. arXiv:1205.1112v1.
\bibitem{gf}B.-N. Guo, F. Qi, Properties and applications of a functions involving exponential functions, Commun. Pure Appl. Anal. 8 (2009).
\bibitem{34} M. E. H. Ismail and M. E. Muldoon, Inequalities and monotonicity 
properties for gamma and q-gamma functions, available online at 
http://arxiv.org/abs/1301.1749.
\bibitem{ks} C. Krattenthaler, H.M. Srivastava, Summations for basic hypergeometric series involving a q-analogue of the
digamma function, Computers and Mathematics with Applications 32 (2) (1996) 73–91.
\bibitem{ki}C.H. Kimberling, A probabilistic interpretation of complete monotonicity, Aequationes Math. 10 (1974) 152–164.
\bibitem{cpc} A.-J. Li, C.-P. Chen, Some completely monotonic functions involving the Gamma and Polygamma functions, J. Korean Math. Soc. 45 (2008), No. 1, pp. 273–287.
\bibitem{kha1} K. Mehrez, Redheffer type inequalities for modified Bessel functions, Arab. J. Math. Sci, 2015; avaible online at: http://dx.doi.org/10.1016/j.ajmsc.2015.03.001
\bibitem{merkle} M. Merkle, "On log-convexity of a ratio of gamma functions," Univ. Beograd. Publ. Elektrotehn. Fak. Ser. Mat.
8 (1997), 114–119.
\bibitem{salem} A. Salem, Ceratin class of approximation for the $q-$digamma function,Rocky Mountain Journal of Mathematics, ( To appear).
\bibitem{salem4} A. Salem, Some Properties and Expansions Associated with q-Digamma Func- tion, Quaestiones math-
ematicae, (To appear)
\bibitem{wei} C. Wei, Q. Gu, q-Generalizations of a family of harmonic number identities,
Advances in Applied Mathematics, Vol.45 (2010) 24-27.
\end{thebibliography}
\end{document}